\documentclass[12pt]{amsart}
\usepackage{amssymb, latexsym, amsmath, amsthm, verbatim}
\usepackage{fullpage}
\usepackage[pdftex]{graphicx}

\newtheorem{theorem}{Theorem}
\newtheorem{lemma}[theorem]{Lemma}
\newtheorem{corollary}[theorem]{Corollary}
\newtheorem{Conditions}{Conditions} 
\newtheorem*{thm}{Theorem}

\newcommand{\R}{\mathbf{R}}
\renewcommand{\S}{\mathbf{S}}

\begin{document}
 
\title{Width and Flow of Hypersurfaces by Curvature Functions}

\author{Maria Calle}
\address{Max Planck Institute for Gravitational Physics (Albert Einstein Institute)\\
Am M\"{u}hlenberg 1\\
14476 Golm\\
Germany} 

\author{Stephen J. Kleene}
\address{Department of Mathematics\\
Johns Hopkins University\\
3400 N. Charles St.\\
Baltimore, MD 21218}

\author{Joel Kramer}
\address{Department of Mathematics\\
Johns Hopkins University\\
3400 N. Charles St.\\
Baltimore, MD 21218}

\email{maria.calle@aei.mpg.de, skleene@math.jhu.edu and jkramer@math.jhu.edu}

\maketitle

It is well known that a convex closed hypersurface in $\R^{n+1}$ evolving by its mean curvature remains convex and vanishes in finite time.  In \cite{andrews1}, Ben Andrews showed that the same holds for much broader class of evolutions. T. H. Colding and W. P. Minicozzi in \cite{colding-minicozzi}
give a bound on extinction time for mean curvature flow in terms of an invariant of the initial hypersurface that they call the width.  In this paper, we generalize this estimate to the class evolutions considered by Andrews:

\begin{thm} \label{mainresult}
Let $\{M_t\}_{t\ge 0}$ be a one-parameter family of smooth compact and strictly convex hypersurfaces in $\R^{n+1}$ satisfying $\frac{d}{dt}M_t = F\nu_t$, where $\nu_t$ is the unit normal of $M_t$ and where $F$ satisfies conditions 3.1 in \cite{andrews1}. Let $W(t)$ denote the width of the hypersurface $M_t$. Then in the sense of limsup of forward difference quotients it holds:
\begin{equation} \label{first}
\frac{d}{dt}W\le-\frac{4\pi}{C_0n},
\end{equation}
and
\begin{equation} 
W(t)\le W(0)-\frac{ 4\pi}{C_0n}t. 
\end{equation}
\end{thm}

 All functions $F$ in this class are assumed to be either concave or convex. In the concave case, we require an a priori pinching condition on the initial hypersurface $M_0$. In the case of convexity a pinching condition is not needed, and we can take $C_0 = 1$ above.  This class is of interest since it contains many classical flows as particular examples, such as the n-th root of the Gauss curvature, mean curvature, and hyperbolic mean curvature flows. The key to arriving at our estimate on extinction time is a uniform estimate on the rate of change of the width, for which preservation of convexity of the evolving hypersurface is fundamental.

We also study flows which are similar to those considered by Andrews in \cite{andrews1}, except that we allow for higher degrees of homogeneity.  The motivation for this consideration is that the degree 1 homogeneity condition on the flows given in Andrews' paper excludes some naturally arising evolutions, such as powers of mean curvature, for which preservation convexity and extinction time estimates are known (see \cite{schulze}).  An analogous result as (\ref{first}) holds for the time zero derivative of width of an initially convex hypersurface for these kinds of flows.  However in the general case it is, to our knowledge, unknown as to whether convexity is preserved, or even if there is a well-defined extinction time.  Consequently, nothing can be said about the long term behavior of the width of a hypersurface evolving under these flows. We conclude with a formulation of the problem in terms of a higher dimensional analogue of the width, the 2-width, and show that an analogous estimate on the derivative of the 2-width holds, and that in this case the convexity assumption can be loosened to 2-convexity.

The authors would like to thank Professor Minicozzi for bringing the problem to our attention and his continued guidance.

Throughout this paper, for any manifold $M$ immersed in $\R^{n+1}$, $H_M$ will denote the mean curvature vector, i.e., the trace of the Wiengarten map.  We will use $Energy(\cdot)$ to denote the energy of a $W^{1,2}$ map into $\R^m$. Also, when needed, we will assume that the manifold $M$ is oriented  so that the prinicipal curvatures of the convex immersion are positive.

\section{contraction of convex hypersurfaces in $\R^{n + 1}$}

Let $M^n$ be a smooth, closed manifold of dimension $n\ge 2$.  Given a convex smooth immersion $\phi_0:M^n\to\R^{n+1}$, we consider a smooth family of immersions $\phi:M^n\times[0,T)\to\R^{n+1}$ satisfying an equation of the following form:

\begin{eqnarray} \label{evolution}
\frac{\partial}{\partial t}\phi(x,t) & = & F(\lambda(x,t))\nu(x,t) \\
\phi(x,0) & = & \phi_0(x) 
\end{eqnarray}

for all $x\in M^n$ and $t\in[0,T)$.  In this equation $\nu(x,t) = H_{M_t}/|H_{M_t}|$ is the inward pointing unit normal to the hypersurface $M_t:=\phi(M^n,t)=\phi_t(M^n)$ at the point $\phi_t(x)$, and $\lambda(x,t)=(\lambda_1(x,t),...,\lambda_n(x,t))$ are the principal curvatures of $M_t$ at $\phi_t(x)$, that is, the eigenvalues of the Weingarten map of the immersion at that point.

We consider velocity functions $F$ satisfying the following hypotheses, taken from \cite{andrews1}:

\begin{Conditions}\label{conditions}
\begin{enumerate}
\item $F:\Gamma_+\to\R$ is defined on the positive cone $\Gamma_+=\{\lambda=(\lambda_1,...,\lambda_n)\in\R^n:\lambda_i>0,i=1 \ldots n\}$.
\item $F$ is a smooth symmetric function.
\item $F$ is strictly increasing in each argument: $\frac{\partial F}{\partial \lambda_i}>0$ for $i=1\ldots n$ at every point in $\Gamma_+$.
\item $F$ is homogeneous of degree one: $F(c\lambda)=cF(\lambda)$ for any positive $c\in\R$.
\item $F$ is strictly positive on $\Gamma_+$, and $F(1,\ldots ,1)=1$.
\item one of the following holds:
	\begin{enumerate}
	\item F  is convex; or \label{F_is_convex}
	\item F is concave and either \label{F_is_concave}
		\begin{enumerate}
		\item $n = 2$;
		\item f approaches zero on the boundary of $\Gamma_+$; or
		\item $\sup_{x \in M} \frac{H}{F} < \liminf_{\lambda \in \partial \Gamma_+} \frac{\Sigma\lambda_i}		{F(\lambda)}$
		\end{enumerate}
	\end{enumerate}	 
\end{enumerate}
\end{Conditions}

Given an initially convex smooth immersion $\phi_0:M^n\to\R^{n+1}$ and a function $F$ satisfying conditions \ref{conditions} , B. Andrews proved (in \cite{andrews1}, Corollary 3.6, Theorem 4.1 and Theorem 6.2 )  that there exists a unique family of smooth immersions $\phi(x,t): M^n \times [0, T) \rightarrow \R^{n + 1}$ satisfying equations (\ref{evolution}) for all $x\in M^n$. Moreover, he proved that the embeddings $M_t$ remain convex and converge uniformly to a point as $t \rightarrow T$.

\section{Width and the existence of Good Sweepouts} \label{Width}

In this section we recall some notions of \cite{colding-minicozzi}.

Convexity of the immersion $\phi_t: M \rightarrow M_t \subset \R^{n+1}$ implies that $M$ is diffeomorphic to the n-sphere $\S^n$, via the Gauss map.  $\S^n$ is then equivalent to $\S^1 \times B^{n-1}/\sim$, where $\sim$ is the equivalence relation $(\theta_1,y)\sim(\theta_2,y)$ where $\theta_1,\theta_2\in \S^1$ and $y\in\partial B^{n-1}$.  Here $B^{n -1}$ is the unit ball in $\R^{n-1}$.  We use this decomposition of $M  = \S^n$ to define the width $W(t)$ of the immersion $\phi_t$. 

Take $\mathcal{P} = B^{n-1}$, and define $\Omega_t$ to be the set of continuous maps $\sigma: \S^1 \times \mathcal{P} \rightarrow M_t$ such that for each $s  \in \mathcal{P}$ the map $\sigma(\cdot, s)$ is in $W^{1,2}$, such that the map $s \rightarrow \sigma(\cdot,s)$ is continuous from $\mathcal{P}$ into $W^{1,2}$, and finally such that $\sigma(\cdot,s)$ is a constant map for all $s \in \partial \mathcal{P}$.  We will refer to elements of the set $\Omega_t$ as ``sweepouts'' of the manifold $M_t$.  Given a sweepout  $\hat{\sigma} \in \Omega_t$ representing a non-trivial homotopy class in $\pi_n{M}$, the homotopy class $\Omega_t^{\hat{\sigma}}$ is defined to be the set of all maps $\sigma$ that are homotopic to $\hat{\sigma}$ through $\Omega_t$.  We then define the width, $W(t)$, as follows: Fix a sweepout $\beta \in \Omega_0$ representing a non-trivial homotopy class in $M$ and let $\beta_t \in \Omega_t$ denote the corresponding sweepout of $M_t$.  We then take
\begin{equation}
W(t) = W(t,\beta) = \inf_{\sigma \in \Omega_t^{\beta_t}}\max_{s \in \mathcal{P}} Energy(\sigma(\cdot, s)) \notag
\end{equation}
We note that the width is always positive until extinction time. 

In the proof of our main theorem, we will rely heavily on the following
\begin{lemma} \label{goodsweepouts}
For each $t$, there exists a family of sweepouts $\{\gamma^j\}  \subset \Omega_t$ of $M_t$ that satisfy the following:
\begin{enumerate}  
\item The maximum energy of the slices $\gamma^j(\cdot,s)$ converges to $W(t)$.
\item For each $s \in \mathcal{P}$, the maps $\gamma^j(\cdot,s)$ have Lipschitz bound $L$ independent of both $j$ and $s$.
\item Almost maximal slices are almost geodesics: That is, given $\epsilon > 0$, there exists $\delta > 0 $ such that if $j > 1/\delta$ and $Energy(\gamma^j(\cdot,s)) > W(t) - \delta$ for some $s$, then there exists a non-constant closed geodesic $\eta$ in $M_t$ such that $dist(\eta, \gamma^j(\cdot,s)) < \epsilon.$ 
\end{enumerate}
\end{lemma}

The existence of such a family of sweepouts is established in \cite{colding-minicozzi}, Theorem 1.9, and we will not include the proof here.  The distance in the third statement of the previous Lemma is given by the $W^{1,2}$ norm on the space of maps from $\S^1$ to $M$.

\section{F convex} \label{Fconvex}

We assume the velocity function $F$ satisfies conditions (1) through (5)  above, in addition to condition ($\ref{F_is_convex}$).

We then have the following:

\begin{lemma} \label{lemma1}
For all $(\lambda_1,...,\lambda_n)\in\Gamma_+$, it holds $$F(\lambda_1,...,\lambda_n)\ge\frac{\sum_1^n\lambda_i}{n}.$$
\end{lemma}

\begin{proof}
Let $\lambda=(\lambda_1,...,\lambda_n)$.  Let $S_n$ the group of permutations of $(1,...,n)$, and let $\lambda_{\sigma}=(\lambda_{\sigma(1)},...,\lambda_{\sigma(n)})$ for each $\sigma\in S_n$.  Then, since $F$ is symmetric, $F(\lambda_{\sigma})=F(\lambda)$ for all $\sigma\in S_n$.

Let $h=\frac{\sum_1^n\lambda_i}{n}$, let $\tilde{h}=(h,...,h)\in\Gamma_+$.  The we have the following:

\begin{eqnarray}
\sum_{\sigma\in S_n}\lambda_{\sigma} & = & (\sum_{\sigma\in S_n}\lambda_{\sigma(1)},...,\sum_{\sigma\in S_n}\lambda_{\sigma(n)}) \notag \\
& = & |S_{n-1}|(\sum_1^n\lambda_i,...,\sum_1^n\lambda_i) \\
& = & n|S_{n-1}|\tilde{h} = |S_n|\tilde{h},\notag
\end{eqnarray}
where the last inequality follows from $|S_n|=n!=n(n-1)!=n|S_{n-1}|$.  Then, since $\frac{1}{|S_n|}<1$, by convexity we have:

\begin{equation}
F(\tilde{h})\le\frac{1}{|S_n|}\sum_{\sigma\in S_n}F(\lambda_{\sigma})=\frac{|S_n|}{|S_n|}F(\lambda)=F(\lambda). 
\end{equation}
But then, by the forth and fifth property of $F$ we have that $h=hF(1,...,1)=F(\tilde{h})$, and therefore we obtain the result.
\end{proof}

We can then prove the following estimate (corollary 2.9 in \cite{colding-minicozzi}):

\begin{lemma} \label{lemma2}
Let $\Sigma\subset M_0$ be a closed geodesic and $\Sigma_t\subset M_t$ the corresponding evolving closed curve.  Let $E_t$ be the energy of $\Sigma_t$.  Then:
\begin{equation}
\frac{d}{dt}_{t=0}E_t\le-\frac{4\pi}{n}. 
\end{equation}
\end{lemma}

\begin{proof}
Observe that, since $M_0$ is convex and $\Sigma$ is minimal, $H_\Sigma$ points in the same direction as $\nu$.  By convexity, we have that $|H_\Sigma |\le |H_{M_0}|$.  Let $V_t$ be the length of the closed curve $\Sigma_t$.  Then, by the first variation formula for volume (see 9.3 and 7.5' in \cite{simon}) we have the following:
\begin{eqnarray}\label{eqn1}
\frac{d}{dt}_{t=0}V_t & = & -\int_{\Sigma}\langle H_{\Sigma},F(x,t)\nu(x,t)\rangle \notag \\
& \le & -\frac{1}{n}\int_{\Sigma}|H_\Sigma|| H_{M_0}| \\ 
& \le & -\frac{1}{n}\int_{\Sigma}|H_{\Sigma}|^2. \notag
\end{eqnarray}
Here the first inequality follows from Lemma \ref{lemma1}, and the second inequality follows from $0\le |H_{\Sigma}| \le |H_{M_0}|$.
Then we can compute the variation of energy as follows:
\begin{equation}
\pi\frac{d}{dt}_{t=0}E_t=V_0\frac{d}{dt}_{t=0}V_t\le-\frac{V_0}{n}\int_{\Sigma}|H_{\Sigma}|^2\le-\frac{1}{n}\left(\int_{\Sigma}|H_{\Sigma}|\right)^2\le-\frac{4\pi^2}{n}. 
\end{equation}
Here the first inequality follows from (\ref{eqn1}), the second from the Cauchy-Schwarz inequality, and the last inequality follows since by Borsuk-Fenchel's theorem every closed curve in $\R^{n+1}$ has total curvature at least $2\pi$ (see \cite{borsuk}, \cite{fenchel}). 
\end{proof}

\section{F concave} \label{Fconcave}
We assume the velocity function $F$ satisfies conditions (1) through (5)  above, in addition to condition ($\ref{F_is_concave}$). We also require that the initial convex immersion $\phi_0: M \rightarrow \R^{n+1}$ satisfies an a priori pinching condition

\begin{equation}
max\{\lambda_1(x,0), ..., \lambda_n(x,0) \} \leq C_0min\{\lambda_1(x,0), ...,\lambda_n(x,0)\}
\end{equation}

for all $x \in M$. Here we have chosen an orientation on $M$ so that the sign of the principle curvatures is positive.  We then have 

\begin{equation}
\sup_{x\in M_0}\frac{|H_{M_0}|}{nF}=C_0. 
\end{equation}
We then have the following:

\begin{lemma} \label{lemma3}
For all $t\ge 0$, $$\sup_{x\in M_t}\frac{|H_{M_t}|}{nF}\le\sup_{x\in M_0}\frac{|H_{M_0}|}{nF}=:C_0.$$ 
\end{lemma}

\begin{proof}
By the parabolic maximum principle.  See proof of Theorem 4.1 in \cite{andrews1}.
\end{proof}

We can then prove the following estimate:

\begin{lemma} \label{lemma4}
Let $\Sigma\subset M_0$ be a closed geodesic and $\Sigma_t\subset M_t$ the corresponding evolving closed curve.  Let $E_t$ be the energy of $\Sigma_t$.  Then:
\begin{equation}
\frac{d}{dt}_{t=0}E_t\le-\frac{4\pi}{nC_0}. 
\end{equation}
\end{lemma}

\begin{proof}

The proof follows exactly as in the proof of Lemma \ref{lemma2}, except that here Lemma \ref{lemma3}
gives the inequality
\begin{equation}
\frac{d}{dt}_{t =0} V_t \leq - \frac{1}{nC_0}\int_\Sigma |H_\Sigma|^2 
\end{equation} 
We then have as before 
\begin{equation}
\pi\frac{d}{dt}_{t=0}E_t=V_0\frac{d}{dt}_{t=0}V_t\le-\frac{V_0}{nC_0}\int_{\Sigma}|H_{\Sigma}|^2\le-\frac{1}{nC_0}\left(\int_{\Sigma}|H_{\Sigma}|\right)^2\le-\frac{4\pi^2}{nC_0}.
\end{equation}

\section{Extinction Time} \label{Extinctiontime}
\end{proof}

We can now prove our main theorem:

\begin{theorem} \label{theorem1}
Let $\{M_t\}_{t\ge 0}$ be a one-parameter family of smooth compact and strictly convex hypersurfaces in $\R^{n+1}$ flowing by equation (\ref{evolution}).  Let $C_0=1$ if $F$ is convex, and $C_0$ as in Lemma~\ref{lemma3} if $F$ is concave.  Then in the sense of limsup of forward difference quotients it holds:
\begin{equation} \label{widthderivative}
\frac{d}{dt}W\le-\frac{4\pi}{nC_0},
\end{equation}
and
\begin{equation} \label{extinctiontime}
W(t)\le W(0)-\frac{4\pi}{nC_0}t.
\end{equation}
\end{theorem}

As a consequence, we get the following bound on the extinction time:

\begin{corollary}
Let $\{M_t\}_{t\ge 0}$ be as above, then it becomes extinct after time at most $$\frac{nC_0W(0)}{4\pi}.$$
\end{corollary}

Observe that in the concave case, the constant $C_0$ depends on the pinching of the initial hypersurface $M_0$, whereas in the convex case the bound on the extinction time is independent of the evolving hypersurface.

We will need the following consequence of the first variation formula for energy in the proof of our main theorem: If $\sigma_t$, $\eta_t$ are two families of curves evolving by a $C^2$ vector field $\mathbf{V}$, then 
\begin{equation} \label{twocurves}
 |\frac{d}{dt}Energy(\eta_t) - \frac{d}{dt}Energy(\sigma_t)| \leq C||\mathbf{V}||_{C^2}||\sigma_t - \eta_t ||_{W^{1,2}}(1 + \sup|\sigma'_t|^2)
\end{equation}

\begin{proof}[Proof of Theorem~\ref{theorem1}]
Here we follow the outline of Theorem 2.2 in \cite{colding-minicozzi}

Fix a time $\tau$.  Throughout the proof, $C$ will denote  a constant depending only on $M_\tau$, but will be allowed to change from inequality to inequality.  Let $\gamma^j$ be a sequence of sweepouts in $M_\tau$ given by Lemma \ref{goodsweepouts}.  For $t \geq \tau$, let $\sigma^j_s(t)$ denote the curve in $M_t$ corresponding to $\gamma^j_s=\gamma^j(\cdot,s)$, and set $e_{s,j}(t) = Energy(\sigma^j_s(t))$.  We will use the following claim to establish an upper bound for the width: Given $\epsilon > 0$, there exist $\delta > 0 $ and $h_0 > 0$ so that if $j > 1/\delta$ and $0 < h < h_0$, then for all $s \in \mathcal{P}$

\begin{equation} \label{claim}
e_{s,j}(\tau + h) - \max_{s_0} e_{s_0,j}(\tau) \leq \left[\frac{-4\pi}{nC_0} + C\epsilon\right]h + Ch^2.
\end{equation}

The result follows from (\ref{claim}) as follows: take the limit as $j \rightarrow \infty$ in (\ref{claim}) to get 
\begin{equation} \label{claimlimited}
\frac{W(\tau + h) - W(\tau)}{h} \leq \frac{-4\pi}{nC_0} + C\epsilon +  Ch.
\end{equation}

Taking $\epsilon \rightarrow 0$ in (\ref{claimlimited}) gives (\ref{widthderivative}). 

It remains to establish (\ref{claim}).  First, let  $\delta > 0$ be given by Lemma \ref{goodsweepouts}.  Since $\beta$ is homotopically non-trivial in $M_\tau$, $W(\tau)$ is positive and we can assume that $\epsilon^2 < W(\tau)/3$ and $\delta < W(\tau)/3$.  If $j > 1/\delta$, and $e_{s, j}(\tau) > W(\tau) - \delta$, then Lemma \ref{goodsweepouts} gives a non-constant closed geodesic $\eta$ in $M_\tau$ with $dist(\eta, \gamma^j_s) < \epsilon$.  Letting $\eta_t$ denote the image of $\eta$ in $M_t$, we have, using (\ref{twocurves}) with $\mathbf{V} = F\nu$ and the uniform Lipschitz bound $L$ for the sweepouts  at time  $\tau$, that 
\begin{equation} \label{uniformbound}
\frac{d}{d t}_{t = \tau} e_{s,j}(t) \leq \frac{d}{d t}_{t = \tau} Energy(\eta_t) + C\epsilon||F\nu||_{C^2}(1 + L^2) \leq \frac{-4\pi}{nC_0} + C\epsilon
\end{equation}  
Since $\sigma_s^j(t)$ is the compositon of $\gamma_s^j$ with the smooth flow and $\gamma_s^j$ has Lipschitz bound $L$ independent of $j$ and $s$, the energy function $e_{s,j}(\tau + h)$ is a smooth function of $h$ with a uniform $C^2$ bound independent of both $j$ and $s$ near $h = 0$.  In particular, the Taylor expansion for $e_{s,j}(\tau + h)$ gives
\begin{equation}
e_{s,j}(\tau + h) - e_{s,j}(\tau) = \frac{d}{d h}e_{s,j}(\tau)h + R_1(e_{s,j}(\tau + h))
\end{equation}
where $R_1(e_{s,j}(\tau +h))$ denotes the first order remainder.  The uniform $C^2$ bounds on $e_{s,j}(t)$ give uniform bounds on the remainder terms $R_1(e_{s,j}(\tau + h))$, and using (\ref{uniformbound}), we get 
\begin{equation}
e_{s,j}(\tau + h) - e_{s,j}(\tau) \leq \left\{ \frac{-4\pi}{nC_0}+  C\epsilon\right \}h + Ch^2
\end{equation}
from which the claim follows.  In the case $e_{s, j} (\tau) \leq W(\tau) - \delta$, the claim (\ref{claim}) automatically holds after possibly shrinking $h_0 > 0$:

\begin{equation}
e_{s,j}(\tau  + h) - \max_{s_0}e_{s_0,j}(\tau) \leq e_{s,j}(\tau + h) - e_{s,j}(\tau) - \delta
\end{equation}
Taking $h$ sufficiently small, so that $-\delta \leq \frac{-4\pi}{nC_0} h$, and using the differentiability of $e_{s,j}(t)$, we get (\ref{claim}). 
To get (\ref{extinctiontime}), note that for any $\epsilon > 0$, the set $\{ t | W(t) \leq W(0) - (\frac{4\pi}{nC_0} - \epsilon)t\}$ contains $0$, is closed, and (\ref{widthderivative}) implies that is also open: Take $F(t) = W(t) - W(0) + (\frac{4\pi}{nC_0} - \epsilon)t$, and note that $\frac{d}{d t} F(t)< 0$ in the sense of limsup of forward difference quotients.  We therefore have $W(t) \leq W(0) - (\frac{4\pi}{nC_0} - \epsilon)t $ for all t up to extinction time.  Taking $\epsilon \rightarrow 0$ gives (\ref{extinctiontime}).  

\end{proof}

\section{F homogeneous of degree k} \label{Fkhomogeneous}
We assume the function $F$ has the same properties as in section 3, except that it is homogeneous of degree $k$ for some integer $k>0$, i.e., $F(c\lambda_1,...,c\lambda_n)=c^kF(\lambda_1,...,\lambda_n)$.   Then we can prove the following lemma:

\begin{lemma}
For all $(\lambda_1,...,\lambda_n)\in\Gamma_+$, it holds $$F(\lambda_1,...,\lambda_n)\ge\left( \frac{\sum_1^n\lambda_i}{n} \right)^k.$$
\end{lemma}

\begin{proof} \label{lemma6}
Analogous to the proof of Lemma~\ref{lemma1}.
\end{proof}

We can then prove the following estimate (analogous to lemma \ref{lemma2}):

\begin{lemma} \label{lemma7}
Let $\Sigma\subset M_0$ be a closed geodesic and $\Sigma_t\subset M_t$ the corresponding evolving closed curve.  Let $E_t$ be the energy of $\Sigma_t$.  Then:
\begin{equation}
\frac{d}{dt}_{t=0}E_t^{\frac{k+1}{2}}\le-\frac{k+1}{n^k}(2\pi)^{\frac{k+1}{2}}.
\end{equation}
\end{lemma}

\begin{proof}
Observe that, since $M_0$ is convex and $\Sigma$ is minimal, the mean curvature vector of $\Sigma$ in $\R^{n+1}$ points in the same direction as $\nu$.  We have that $|H_\Sigma| \le |H_{M_0}|$ by convexity.  Let $V_t$ be the length of the close curves $\Sigma_t$, we can use the first variation formula for volume (see 9.3 and 7.5' in \cite{simon}) to obtain the following:
\begin{align}\label{eqn6}
\frac{1}{k+1}\frac{d}{dt}_{t=0}V_t^{k+1}  = &  -V_0^k\int_{\Sigma}\langle H_{\Sigma},F(x,t)\nu(x,t)\rangle  \le -\frac{1}{n^k}V_0^k\int_{\Sigma}|H_\Sigma||H_{M_0}|^k \\ 
& \le  -\frac{1}{n^k}V_0^k\int_{\Sigma}|H_\Sigma|^{k+1} 
 \le  -\frac{1}{n^k}\left(\int_{\Sigma}|H_\Sigma|\right)^{k+1}   \le  -\frac{1}{n^k}(2\pi)^{k+1}. \notag
\end{align}
\begin{equation}
\end{equation}
Here the first inequality follows from Lemma~\ref{lemma1}, the second inequality follows from $0\le|H_\Sigma|\le |H_{M_0}|$, the third one follows from H\"{o}lder's inequality and the last inequality follows as above from Borsuk-Fenchel's theorem.
Then we can compute the variation of energy as follows:
\begin{equation}
\frac{d}{dt}_{t=0}E_t^{\frac{k+1}{2}}=\frac{1}{(2\pi)^{\frac{k+1}{2}}}\frac{d}{dt}_{t=0}V_t^{k+1}\le-\frac{k+1}{n^k}(2\pi)^{\frac{k+1}{2}}.
\end{equation}
\end{proof}

Finally, we could use this estimate to give a bound on the decrease rate of the width as in Theorem~\ref{theorem1}, to obtain the following:

\begin{theorem} \label{theorem4}
Let $\{M_t\}_{t\ge 0}$ be a one-parameter family of smooth compact and strictly convex hypersurfaces in $\R^{n+1}$ flowing by equation (\ref{evolution}), with $F$ as in this section.  Then in the sense of limsup of forward difference quotients it holds:
\begin{equation}
\frac{d}{dt}W^{k+1}\le-\frac{k+1}{n^k}(2\pi)^{\frac{k+1}{2}},
\end{equation}
and
\begin{equation}
W(t)\le W(0)-\frac{k+1}{n^k}(2\pi)^{\frac{k+1}{2}}t
\end{equation}
for as long as the evolving manifold remains smooth and strictly convex.
\end{theorem}

\section{2 width and 2-convex manifolds} \label{2width}

Until now, we have assumed that the evolving manifold is strictly convex.  Also, we have defined the width using sweepouts by curves.  We can generalize the result for manifolds which are ``2-convex'', that is, such that the sum of any two principal curvatures is positive, by using a different width, defined by sweepouts of 2-spheres.  The fundamental point is that in the case of 2-width, as earlier, one can rely on the existence of a sequence of ``good sweepouts'' as comparisons in order to estimate the derivative of the width.  For higher dimensional sweepouts the analogous result is not known, and so we cannot argue as before in trying to prove an extinction time estimate.  A motivation for dealing with general higher dimensional widths, as opposed to 1-width, is that it allows for a relaxation of the convexity condition.

In defining the width above, we chose to decompose $\S^n$ topologically into the space $\S^1 \times B^{n-1}$, with each set $\{(t,s):t\in\S^1\}$ collapsed to a point, for each $s\in\partial B^{n-1}$.  The purpose of this decomposition is that, since $M$ is a closed convex hypersurface, it is topologically $\S^n$, and the decomposition ensures that the sweepouts induce non-trivial maps on $\pi_n(\S^n)$ - i.e. they are not homotopically trivial - and hence that the width is positive until the hypersurface becomes extinct. 
If we require that the initial immersion $\phi_0: M \rightarrow \R^{n+1} $ is only 2-convex, we lose information about the global topology of $M$, and so it makes sense to use more general parameter spaces than those that ensure the sweepouts are homotopy $\S^n$'s.

Let $\mathcal{P}$ be a compact finite dimensional topological space, and let $\Omega_t$ be the set of continuous maps $\sigma:\S^2\times\mathcal{P} \rightarrow M$ so that for each $s\in\mathcal{P}$ the map $\sigma(\cdot,s)$ is in $C^0\cap W^{1,2}(\S^2,M)$, the map $s\to\sigma(\cdot,s)$ is continuous from $\mathcal{P}$ to $C^0\cap W^{1,2}(\S^2,M)$, and finally $\sigma$ maps $\partial\mathcal{P}$ to point maps.  Given a map $\hat{\sigma}\in\Omega_t$, the homotopy class $\Omega_t^{\hat{\sigma}}\subset\Omega_t$ is defined to be the set of maps $\sigma\in\Omega_t$ that are homotopic to $\hat{\sigma}$ through maps in $\Omega_t$.  To define the 2-width, fix a sweepout $\beta \in \Omega_0$ representing a non-trivial homotopy class in $M_0$ and let $\beta_t \in \Omega_t$ represent the corresponding sweepout in $M_t$.  Then, for each $t$, we define
\begin{equation}
W_2(t)=W_2(t, \hat{\sigma})=\inf_{\sigma\in\Omega_t^{\hat{\sigma}}}\max_{s\in\mathcal{P}}Energy(\sigma(\cdot,s)),
\end{equation} 
where the energy is given by
\begin{equation}
Energy(\sigma(\cdot,s))=\frac{1}{2}\int_{\S^2}|\nabla_x\sigma(x,s)|^2dx.
\end{equation}
It was shown in \cite{colding-minicozzi2} and \cite{colding-minicozzi3} that we could also define the energy using area, and we would obtain the same quantity:
\begin{equation}
W_2^A=\inf_{\sigma\in\Omega_t^{\hat{\sigma}}}\max_{s\in\mathcal{P}}Area(\sigma(\cdot,s))=W_2.
\end{equation}

We should note that, in the above formulation of the 2-width, it was necessary to assume the existence of a nontrivial homotopy class in M. The 2-width of a torus, for example, is always zero.  

As in the case of $1$-width defined above, we will use the existence of  a sequence of ``good sweepouts'' that approximate the width.  Namely, the following theorem was proven in \cite{colding-minicozzi2}, Theorem 1.14: 

\begin{theorem}
Given a map $\beta \in \Omega_t$ representing a non-trivial class in $\pi_n(M_t)$, there exists a sequence of sweepouts $\gamma^j \in \Omega_t^{\beta}$ with $\max_{s\in\mathcal{P}}Energy(\gamma^j(\cdot,s))\to W_2(\beta)$, and so that given $\epsilon>0$, there exist $\bar{j}$ and $\delta>0$ so that if $j>\bar{j}$ and
\begin{equation}
Area(\gamma^j(\cdot,s))>W_2(\beta)-\delta,
\end{equation}
then there are finitely many harmonic maps $u_i: \S^2 \rightarrow M_t$ with
\begin{equation}
d_V(\gamma^j(\cdot,s),\cup_i\{u_i\})<\epsilon.
\end{equation}
Here $d_V$ denotes varifold distance as defined in \cite{colding-minicozzi2}.
\end{theorem}

We consider now a family $M_t^n\in\R^{n+1}$ evolving by the evolution equation (\ref{evolution}).  However, now the manifolds $M_t$ are not necessarily convex, but they are {\it 2-convex}: we may choose an orientation on the ambient space $\R^{n+1}$ and on $M$ so that the sum of any two principal curvatures is  always positive.  Equivalently, if the principal curvatures are $\lambda_1\le\lambda_2\le...\le\lambda_n$, then $|\lambda_1|\le\lambda_2$.  Observe that this implies that $M_t$ is mean convex, that is, the mean curvature is positive in this orientation.  We assume also that $n>3$.

The velocity function $F$ now is defined in all $\R^n$, and satisfies the rest of conditions in section~\ref{Fconvex}.  Observe that the third condition on $F$ ensures that the evolution equation is parabolic, and therefore a smooth solution exists at least on a short time interval (see Theorem 3.1 in \cite{huisken-polden}).  Also, Lemma~\ref{lemma1} still holds in this case.  Then we can generalize Lemma~\ref{lemma2} as follows:

\begin{lemma} \label{lemma8}
Let $\Sigma\subset M_0$ be a closed minimal surface and $\Sigma_t$ the corresponding surface in $M_t$.  Let $V_t=Area(\Sigma_t)$.  Then:
\begin{equation}
\frac{d}{dt}_{t=0}V_t\le-\frac{16\pi}{nC_0},
\end{equation}
where $C_0$ is $1$ if $F$ is convex, and it is as in Lemma~\ref{lemma3} if $F$ is concave.
\end{lemma}

\begin{proof}
Observe that, since $\Sigma$ is minimal in $M_0$, the mean curvature vector of $\Sigma$ is perpendicular to $M_0$.  Let $k_1$ and $k_2$ be the principal curvatures of $\Sigma$ in $M_0$, that is, $H_{\Sigma}=k_1+k_2$.  Then by $2$-convexity (and because $n>3$), we can chose an orientation on $\R^{n+1}$ and $M$ so that:
\begin{equation}
0\le k_1+k_2\le\lambda_{n - 1}+\lambda_n\le(\lambda_1+\lambda_2)+\lambda_{n-1}+\lambda_n\le\sum_1^n\lambda_i.
\end{equation}
This shows that $H_\Sigma$ and $H_{M_0}$ in fact point in the same direction, with $0\le|H_{\Sigma}|\le|H_{M_0}|$.  As before, we can then compute the first variation of volume as: 
\begin{align}
\frac{d}{dt}_{t=0}V_t  = & -\int_{\Sigma}\langle H_\Sigma,F(x,t)\nu(x,t)\rangle
 \le  -\frac{1}{C_0n}\int_{\Sigma}|H_{\Sigma}||H_{M_0}| 
 \le  -\frac{1}{C_0n}\int_{\Sigma}|H_{\Sigma}|^2
 \le  -\frac{16\pi}{C_0n}, \notag
\end{align} 
where the last inequality is from, e.g. \cite{simon2}.
\end{proof}

Then we can prove the following theorem:

\begin{theorem}
Let $\{M_t\}_{t\ge 0}$ be a one-parameter family of smooth compact $2$-convex hypersurfaces in $\R^{n+1}$ (with $n>3$) flowing as represented in equation (\ref{evolution}).  Then in the sense of limsup of forward difference quotients it holds:
\begin{equation}
\frac{d}{dt}W_2\le-\frac{16\pi}{C_0n}
\end{equation}
and
\begin{equation}
W_2(t)\le W_2(0)-\frac{16\pi}{C_0n}t
\end{equation}
for as long as the solution remains smooth and $2$-convex.
\end{theorem}

This result gives a bound for the time of the first singularity.  Unlike in the convex case, we don't know that the submanifold contracts into a point by the first singularity.  The proof is analogous to the proof of Theorem~\ref{theorem1} (see also \cite{colding-minicozzi2} and \cite{colding-minicozzi3}).

\end{document}